\documentclass{amsart}
\usepackage{amsmath,amsfonts,amssymb}
\usepackage{enumerate,mathrsfs}
\newtheorem{lemma}{Lemma}[section]

\newtheorem{proposition}[lemma]{Proposition}
\newtheorem{theorem}[lemma]{Theorem}

\theoremstyle{definition}

\theoremstyle{remark}

\numberwithin{equation}{section} \numberwithin{table}{section}

\DeclareMathOperator{\var}{\mathrm{var}}

\DeclareMathOperator{\ess}{\mathrm{ess}}

\begin{document}
\title[Measurable Liv\v{s}ic regularity via transfer operators]{General real measurable Liv\v{s}ic regularity via transfer operators}
\author{Ian D. Morris}
\address{School of Mathematical Sciences, Queen Mary University of London, Mile End Road, London E1 4NS, United Kingdom}
\email{i.morris@qmul.ac.uk }

\begin{abstract}
We prove a general measurable Liv\v{s}ic regularity theorem for real-valued cocycles over non-invertible dynamical systems using only abstract hypotheses on an associated transfer operator. As illustrative applications we derive measurable Liv\v{s}ic regularity results in the analytic regularity class for cocycles over real-analytic expanding maps, in the bounded-variation regularity class for $\beta$-transformations, and in $C^\alpha$ regularity over the class of virtually expanding maps recently introduced by M. Tsujii. 
\end{abstract}

\maketitle
\section{Introduction and statement of results}

%fuck it, try Comm Math Phys and see what happens

If $T$ is an ergodic measure-preserving transformation of a probability space $(X,\mathcal{F},\mu)$ then a measurable function $f \colon X \to \mathbb{R}$ is commonly called a \emph{coboundary} if it may be written in the form $f=h \circ T - h $ $\mu$-a.e. for some measurable function $h \colon X \to \mathbb{R}$. The problem of determining whether or not a given function is a coboundary is of key interest in a number of dynamical questions: in the study of statistical limit laws for dynamical systems the coboundary equation typically characterises zero variance, and it is also essential in determining the ergodicity of suspension flows and skew product dynamical systems.  The problem of characterising coboundaries over dynamical systems has therefore received substantial research attention. A.N. Liv\v{s}ic showed in the seminal article \cite{Li71} that if $T \colon X \to X$ is an Anosov diffeomorphism, $\mu$ is absolutely continuous with respect to Lebesgue measure, 
and $f$ is $C^1$ continuous and has the form $f=h\circ T - h$ Lebesgue a.e. where $h \colon X \to \mathbb{R}$ is measurable and essentially bounded, then $f$ may be written in the form $f=g \circ T -g$ where $g$ is also $C^1$. This classical \emph{Liv\v{s}ic regularity theorem} has since been extended in various directions, both by allowing the coboundary $f$ to take values in more general groups and by considering different and broader classes of transformation $T \colon X \to X$ and measure $\mu$, with suitably modified regularity criteria for $f$ (see for example \cite{Bu18,Do05,Go06,Je02,NiPe12,PoWa01,Sa15,Wa00b,ZoCa21}). In \cite{Li71} Liv\v{s}ic also showed that if $T \colon X \to X$ is an Anosov diffeomorphism and $f \colon X \to \mathbb{R}$ is H\"older continuous then $f$ is a coboundary if and only if the ergodic average $\frac{1}{n}\sum_{k=0}^{n-1}f(T^kx)$ is zero for every periodic point $x=T^nx \in X$. Generalisations of this result to coboundaries taking values in more general groups, and to a broader class of base dynamical systems, constitute a parallel and closely-related stream of research (see e.g. \cite{AvKoLi18,GoRo24,Ka11,KoPo16,Sa15,Sh24,Wi13}).

The purpose of this note is to state a mechanism for proving Liv\v{s}ic regularity theorems of the first kind -- in which a measurable coboundary is proved to have additional, stronger regularity properties -- which avoids any overt dependence on the transformation $T$ itself but instead relies on properties of an associated transfer operator, being therefore easily portable across dynamical contexts. While transfer operator approaches to this problem are certainly not new, earlier results such as \cite{Je02,NiPe12,PaPo97,PoYu99} either rely on additional \emph{ad hoc} steps peculiar to each dynamical context or include nontrivial integrability hypotheses on the measurable function $h$ in the coboundary equation $f = h \circ T - h$, when in many contexts only measurability of $h$ is needed. We prove the following very general statement:
\begin{theorem}\label{th:main}
Let $T$ be a measurable transformation of a probability space $(X,\mathcal{F},\mu)$,  $(\mathfrak{X},\|\cdot\|_\mathfrak{X})$ a Banach space of (equivalence classes of) measurable functions from $X$ to $\mathbb{C}$, and $\mathscr{L}$ a linear operator which acts boundedly on $\mathfrak{X}$. Suppose that $T$, $\mathscr{L}$ and $\mathfrak{X}$ satisfy the five properties:
\begin{enumerate}[O1.]
\item\label{it:nonsing}
The transformation $T$ is nonsingular.
\item\label{it:eigen}
There exists a unique $T$-invariant measure which is absolutely continuous with respect to $\mu$. Let $\chi \in L^1(\mu)$ denote the density of this invariant measure.
\item\label{it:embed}
Every element of $\mathfrak{X}$ is integrable with respect to $\mu$, and the embedding $\mathfrak{X} \to L^1(\mu)$ is continuous and injective and has dense range.
\item\label{it:transfer}
The operator $\mathscr{L}$ satisfies 
%\begin{equation}\label{eq:toi}
\[\int (\psi \circ T)\cdot \phi\,d\mu = \int \psi\cdot (\mathscr{L}\phi)\,d\mu\]
%\end{equation}
for every $\psi \in L^\infty(\mu)$ and $\phi \in \mathfrak{X}$.
\item\label{it:fredholm}
The essential spectrum of the operator $\mathscr{L}$ acting on $\mathfrak{X}$ does not include the point $1$. (That is, either $1$ does not belong to the spectrum of $\mathscr{L}$ acting on $\mathfrak{X}$, or it is an isolated point of the spectrum and its associated Riesz projection is of finite rank.)
\end{enumerate}
Let $f \in \mathfrak{X}$ satisfy the two properties:
\begin{enumerate}[F1.]
\item\label{it:cocycle}
There exists a measurable function $h \colon X \to \mathbb{C}$ which is real-valued $\mu$-a.e. and satisfies $f = h \circ T -h$ $\mu$-a.e.
\item\label{it:bdd}
The pointwise multiplication operator $\phi \mapsto f\phi$ is a well-defined bounded linear operator $\mathfrak{X} \to \mathfrak{X}$.
\end{enumerate}
Then $\chi \in \mathfrak{X}$, and there exists $g \in \mathfrak{X}$ such that $g = h\chi$ $\mu$-a.e. In particular if $\phi/\chi \in \mathfrak{X}$ for every $\phi \in \mathfrak{X}$ then we may write $f=\hat{g} \circ T - \hat{g}$ $\mu$-a.e. with $\hat{g}:=g/\chi \in \mathfrak{X}$. 
\end{theorem}
The proof of Theorem \ref{th:main} is presented in the following section, and some examples of applications are indicated in \S\ref{se:ex}. These examples are selected principally for their simplicity, but they include some novelties such as an improved Liv\v{s}ic theorem for the $\beta$-transformation (which removes the integrability hypotheses from the earlier results of \cite{NiPe12,PoYu99}) and a Liv\v{s}ic theorem for the \emph{virtually expanding maps} introduced recently by M. Tsujii in \cite{Ts23}. 

In the applications presented here O\ref{it:fredholm} will follow from the stronger hypothesis that the essential spectral radius of $\mathscr{L}$ on $\mathfrak{X}$ is strictly less than $1$: however, this stronger version of O\ref{it:fredholm} is not itself required. Similarly, in the applications given here it so happens that $\mathfrak{X}$ is a Banach algebra, but F\ref{it:bdd} is clearly much weaker than this hypothesis. It is an interesting open problem to extend Theorem \ref{th:main} to the context of invertible transformations using anisotropic Banach spaces.

\section{Proof of Theorem \ref{th:main}}

\subsection{Fundamentals: action of $\mathscr{L}$ on $L^1(\mu)$}
We first observe that $\mathscr{L}$ defines a bounded linear operator acting on $L^1(\mu)$. Indeed, for every $\phi \in \mathfrak{X}$ and $\psi \in L^\infty(\mu)$ we have
\[\left|\int \psi\cdot ( \mathscr{L}\phi)\,d\mu\right|= \left|\int (\psi \circ T)\cdot\phi \,d\mu \right|\leq \|\psi \circ T\|_{L^\infty(\mu)} \|\phi\|_{L^1(\mu)} \leq  \|\psi\|_{L^\infty(\mu)} \|\phi\|_{L^1(\mu)}\]
using O\ref{it:transfer} followed by O\ref{it:nonsing}. Since $\mathfrak{X}$ is dense in $L^1(\mu)$ by O\ref{it:embed} it follows easily that $\mathscr{L}$ extends to a bounded linear operator on $L^1(\mu)$ with norm $1$, as claimed. 

We next claim that $\mathscr{L}$, acting on $L^1(\mu)$, has a simple eigenvalue at $1$ with eigenfunction $\chi$. Suppose that $\phi \in L^1(\mu)$ is nonzero, real-valued a.e. and satisfies $\mathscr{L}\phi=\phi$, and define $\nu$ to be the signed measure which is absolutely continuous with respect to $\mu$ and has density $\phi$. Since for all $\psi \in L^\infty(\mu)$
\begin{equation}\label{eq:ttt}\int (\psi \circ T)\cdot \phi\,d\mu=\int \psi\cdot(\mathscr{L}\phi)\,d\mu=\int \psi\cdot \phi\,d\mu,\end{equation}
by O\ref{it:transfer}, this measure is $T$-invariant. Decomposing $\nu$ into a difference $\nu_+-\nu_-$ where each of $\nu_+$ and $\nu_-$ is a measure, we see that each of $\nu_+$ and $\nu_-$  is $T$-invariant and absolutely continuous with respect to $\mu$. By O\ref{it:eigen} this implies that one of $\nu_+$ and $\nu_-$ is the zero measure and the other is a nonzero scalar multiple of the measure with density $\chi$, and therefore $\phi$ is a scalar multiple of $\chi$. This conclusion clearly holds also if $\phi \in L^1(\mu)$ is a nonzero and complex-valued eigenfunction by considering separately the real and imaginary parts of $\phi$. On the other hand it follows easily from \eqref{eq:ttt} and the defining property of $\chi$ that $\mathscr{L}\chi=\chi$. We have shown that $1$ is an eigenvalue of $\mathscr{L}$ acting on $L^1(\mu)$ and that $\chi$ spans the corresponding eigenspace. Since $\mathscr{L}$ has norm $1$ it is impossible for there to be a generalised eigenfunction at $1$ since for such a function $\varphi$ the sequence $\mathscr{L}^n\varphi$ would be unbounded. The claim is proved.

We now consider the action of $\mathscr{L}$ on $\mathfrak{X}$. Since for all $\phi \in \mathfrak{X}$
\[\int (\mathscr{L}-I)\phi\,d\mu = \int \mathscr{L}\phi \,d\mu - \int \phi \,d\mu= \int \phi \,d\mu -\int \phi \,d\mu=0\]
using O\ref{it:transfer} with $\psi\equiv 1$, the image of the operator $\mathscr{L}-I$ acting on $\mathfrak{X}$ contains only functions with zero integral. Since $\mathfrak{X}$ embeds densely in $L^1(\mu)$ by O\ref{it:embed} this implies that $\mathscr{L}-I$ cannot be surjective, so $1$ is in the spectrum of $\mathscr{L}$ acting on $\mathfrak{X}$. By O\ref{it:fredholm} this implies that $1$ is an eigenvalue of $\mathscr{L}$ acting on $\mathfrak{X}$. Since $\mathfrak{X}$ embeds injectively into $L^1(\mu)$ by O\ref{it:embed}, this eigenvalue must be simple. It follows that $\chi$ is equal (up to measure zero) to an element of $\mathfrak{X}$, and we identify $\chi$ with this element of $\mathfrak{X}$ henceforth.

\subsection{The perturbed operator and its Riesz projection}
Let $f \in \mathfrak{X}$ be as given by F\ref{it:cocycle}. Let $\mathscr{M}_f \colon \mathfrak{X} \to \mathfrak{X}$ denote the operator corresponding to pointwise multiplication by $f$ (which by F\ref{it:bdd} is a bounded linear operator) and for each $t \in \mathbb{C}$ define an operator $\mathscr{L}_t$ on $\mathfrak{X}$ by
\[\mathscr{L}_t:=\sum_{n=0}^\infty \frac{(it)^n}{n!} \mathscr{L} \mathscr{M}_f^n.\]
The function $t \mapsto \mathscr{L}_t$ is an operator-valued power series in $t$ with coefficients in $\mathcal{B}(\mathfrak{X})$ which converges for all $t \in \mathbb{C}$, which is to say that it defines a holomorphic function $\mathbb{C}\to\mathcal{B}(\mathfrak{X})$. Clearly each $\mathscr{L}_t$ extends to a bounded linear operator on $L^1(\mu)$ which satisfies
 $\mathscr{L}_t\phi \equiv \mathscr{L}(e^{itf}\phi)$ for all $\phi \in L^1(\mu)$.

By O\ref{it:fredholm} together with the preceding results, the operator $\mathscr{L}$ acting on $\mathfrak{X}$ has a simple, isolated eigenvalue at $1$. Let $\Gamma$ be an anticlockwise-oriented circular contour in $\mathbb{C}$ centred at $1$ which does not intersect or enclose any other point of the spectrum of $\mathscr{L}$ acting on $\mathfrak{X}$. There exists a connected open neighbourhood $U$ of $0 \in \mathbb{C}$ such that for every $t \in U$ the  Riesz projection
\[\mathscr{P}_t:=\frac{1}{2\pi i}\int_\Gamma \left(sI-\mathscr{L}_t\right)^{-1}ds\]
is a bounded linear operator on $\mathfrak{X}$ and satisfies $\mathscr{P}_t^2=\mathscr{P}_t$ and $\mathscr{L}_t\mathscr{P}_t=\mathscr{P}_t\mathscr{L}_t\neq0$, and such that for every $t \in U$ the operator $\mathscr{L}_t$ has a unique eigenvalue in the region of $\mathbb{C}$ enclosed by the contour  $\Gamma$ and no other points of the spectrum of $\mathscr{L}_t$ occur in that region. Moreover the function $t \mapsto \mathscr{P}_t$ is a holomorphic function from $U$ to $\mathcal{B}(\mathfrak{X})$, see for example \cite[\S{VII}.1]{Ka95}. Note that the rank of $\mathscr{P}_t$ must be constant with respect to $t \in U$ by continuity since each $\mathscr{P}_t$ is a projection, and since $1$ is a simple eigenvalue the rank of $\mathscr{P}_0$ is $1$, so every $\mathscr{P}_t$ has rank one. By shrinking $U$ if necessary we may assume by continuity that $\mathscr{P}_t\chi \neq 0$ for all $t \in U$. It follows that the function $\chi_t:=\mathscr{P}_t\chi$ is an eigenfunction of $\mathscr{P}_t$, and hence of $\mathscr{L}_t$, for all $t \in U$. Let $\lambda(t)$ be the associated eigenvalue so that $\mathscr{L}_t\chi_t=\lambda(t)\chi_t$ for every $t \in U$.

\subsection{The leading eigenvalue of the perturbed operator}
We claim that $\lambda(t)=1$ for all real $t \in U$. Let $h \colon X \to \mathbb{C}$ be as given in F\ref{it:cocycle}. If $t \in U$ is real, define a linear functional $\ell_t \colon \mathfrak{X} \to \mathbb{C}$ by $\ell_t(\phi):=\int e^{-ith}\phi \,d\mu$. Using O\ref{it:embed}, choose $C>0$ such that $\int|\phi|\,d\mu \leq C \|\phi\|_{\mathfrak{X}}$ for all $\phi \in \mathfrak{X}$. Since $h$ takes real values $\mu$-a.e. by F\ref{it:cocycle} we have
\[|\ell_t(\phi)|\leq \int|e^{-ith}\phi|\,d\mu \leq \int |\phi|\,d\mu \leq C\|\phi\|_{\mathfrak{X}} \]
for all $\phi \in \mathfrak{X}$ and it follows that $\ell_t \colon \mathfrak{X} \to \mathbb{C}$ is continuous. Since $\mathfrak{X}$ embeds densely in $L^1(\mu)$ by O\ref{it:embed} we may choose $\phi \in \mathfrak{X}$ such that $\int|\phi-e^{ith}|d\mu<1$ and therefore
\[|\ell_t(\phi)-1|=\left| \int e^{-ith}\phi\,d\mu-1 \right|=\left| \int e^{-ith}(\phi-e^{ith})\,d\mu \right| \leq \int |\phi-e^{ith}|d\mu<1\]
so that $|\ell_t(\phi)|>0$ and in particular $\ell_t$ is not the zero functional. We next observe that for all $\phi \in\mathfrak{X}$
\[\ell_t(\mathscr{L}_t \phi)=\int e^{-ith}\mathscr{L}\left(e^{itf}\phi\right)\,d\mu = \int e^{-ith\circ T} e^{itf}\phi\,d\mu =\int e^{-ith}\phi\,d\mu=\ell_t(\phi)\]
using O\ref{it:transfer} and F\ref{it:cocycle}, and it follows that $\mathscr{L}_t-I \in \mathcal{B}(\mathfrak{X})$ is not invertible since it takes every element of $\mathfrak{X}$ into the proper subspace $\ker \ell_t$ of $\mathfrak{X}$ and consequently is not surjective. We conclude that $1$ belongs to the spectrum of $\mathscr{L}_t$ for all real $t \in U$. Since for each $t \in U$ there is a unique point of the spectrum of $\mathscr{L}_t$ inside the contour $\Gamma$ we must have $\lambda(t)=1$ for all real $t \in U$  as claimed. 

\subsection{The eigenfunction of the perturbed operator}
We now wish to show that for every $t \in U \cap \mathbb{R}$, the eigenfunction $\chi_t$ is a.e. equal to a scalar multiple of $e^{-ith}\chi$. To begin, we note that for all $\phi \in L^1(\mu)$, $\psi \in L^\infty(\mu)$ and $t \in \mathbb{R}$
\begin{align*}\int \mathscr{L}_t(e^{-ith}\phi)\psi\,d\mu &=\int \mathscr{L}(e^{-ith}e^{itf}\phi)\psi\,d\mu\\
&=\int \mathscr{L}(e^{-ith\circ T}\phi)\psi\,d\mu\\
& = \int \left(e^{-ith \circ T}\phi\right) \psi \circ T\,d\mu= \int e^{-ith}(\mathscr{L}\phi)\psi \,d\mu\end{align*}
using F\ref{it:cocycle}, O\ref{it:transfer} and the definition of $\mathscr{L}_t$. 
Since $\psi \in L^\infty(\mu)$ was arbitrary this demonstrates that $ e^{-ith} \mathscr{L}\phi =\mathscr{L}_t(e^{-ith}\phi)$ for all $\phi \in L^1(\mu)$.  In particular, $\varphi \in L^1(\mu)$ can satisfy $\mathscr{L}_t\varphi=\varphi$ if and only if $\varphi=e^{-ith}\mathscr{L}(e^{ith}\varphi)$, if and only if $e^{ith}\varphi = \mathscr{L}(e^{ith}\varphi)$. Since $\mathscr{L}$ has a $1$-dimensional eigenspace at $1$ when acting on $L^1(\mu)$, we conclude that $\mathscr{L}_t\varphi=\varphi$ if and only if $e^{ith}\varphi$ is a scalar multiple of $\chi$.  For real $t \in U$ the equation $\mathscr{L}_t\chi_t=\lambda(t)\chi_t=\chi_t$ in $\mathfrak{X}$ therefore implies that $\chi_t$ must be proportional to $e^{-ith}\chi$ as an element of $L^1(\mu)$. 
 \subsection{Completion of the proof}
 By the preceding result there  exists a function $c \colon U\cap \mathbb{R} \to \mathbb{C}$ such that $\chi_t = c(t)e^{-ith}\chi$ in $L^1(\mu)$ for all $t \in U \cap \mathbb{R}$. We claim that $c$ is differentiable at zero. For each $\tau>0$ define $X_\tau:=\{x \in X \colon |h(x)| \leq \tau\}$ and let $\mathbf{1}_{X_\tau}$ denote the characteristic function of $X_\tau$. Fix $\tau_0>0$ large enough that $\int_{X_{\tau_0}} \chi d\mu>0$. Clearly $\|h^n \mathbf{1}_{X_{\tau_0}} \|_{L^\infty(\mu)} \leq \tau^n_0$ for every $n \geq 0$, and in particular the function $t \mapsto \int_{X_{\tau_0}} e^{-ith}\chi d\mu = \sum_{n=0}^{\infty} \frac{(it)^n}{n!} \int h^n  \mathbf{1}_{X_{\tau_0}} \chi d\mu$ is analytic with respect to $t$ in a small neighbourhood of $0$ and takes a nonzero value at $t=0$. Since $c(t)=\int_{X_{\tau_0}} \chi_t d\mu / \int_{X_{\tau_0}}e^{-ith}\chi d\mu$ for all real $t \in U$ close enough to $0$ that the denominator is nonzero, the claim follows via the holomorphicity of $t \mapsto \chi_t = \mathscr{P}_t\chi$.
We may now complete the proof of the theorem. Define
\[\chi'_0:=\lim_{t \to 0} \frac{\chi_t-\chi}{t} \in \mathfrak{X}\]
which is well-defined since $t \mapsto \mathscr{P}_t$ is holomorphic. We clearly have
\[\lim_{t \to 0} \frac{1}{t}\left(\int_{X_\tau} \chi_t\psi \,d\mu - \int_{X_\tau}\chi \psi\,d\mu \right)  = \int_{X_\tau}\chi_0'\psi\,d\mu \]
for every $\tau>0$ and $\psi \in L^\infty(\mu)$ by continuity of the linear functional $\phi \mapsto  \int_{X_\tau} \phi\psi \,d\mu$ from $\mathfrak{X}$ to $\mathbb{C}$, where the limit is taken with respect to $t \in U \cap \mathbb{R}$. On the other hand, by the dominated convergence theorem
\[\lim_{t \to 0} \frac{1}{t}\left(\int_{X_\tau} c(t)e^{-ith}\chi \psi \,d\mu - \int_{X_\tau}c(0)\chi \psi\,d\mu \right)  = \int_{X_\tau}(c'(0)-ih)\chi\psi\,d\mu \]
for every $\tau>0$. Hence
\[\int_{X_\tau} \chi_0'\psi\,d\mu = \int_{X_\tau} (c'(0)-ih)\chi \psi\,d\mu\]
for all $\tau>0$ and $\psi \in L^\infty(\mu)$. Fixing $\tau>0$ the arbitrariness of $\psi \in L^\infty(\mu)$ yields $\chi_0'(x)= (c'(0)-ih(x))\chi(x)$ for $\mu$-a.e. $x \in X_\tau$, and since $\tau>0$ is also arbitrary we deduce that $\chi_0'=(c'(0)-ih)\chi$ $\mu$-a.e. Thus $h\chi=-ic'(0)\chi+i\chi_0'$ $\mu$-a.e. and the proof is complete.

%SECTION BREAK
%SECTION BREAK
%SECTION BREAK
%SECTION BREAK
%SECTION BREAK
%SECTION BREAK
%SECTION BREAK
%SECTION BREAK
%SECTION BREAK

\section{Examples}\label{se:ex}

In this section we present some simple applications of Theorem \ref{th:main}, which are chosen primarily for their brevity.
\subsection{The $\beta$-transformation and bounded-variation cocycles}
%Ok, let's use Baladi's Theorem 3.2 and succeeding remarks.

For $f \colon [0,1] \to \mathbb{C}$ we recall that the \emph{variation} of $f$ is defined to be the quantity
\[|f|_{\var}:=\sup\left\{\sum_{i=1}^n |f(x_i)-f(x_{i+1})| \colon 0 \leq x_1 <x_2<\cdots <x_{n+1}\leq1\text{ and }n \geq 1\right\}\]
and we say that $f$ is of \emph{bounded variation} if $|f|_{\var}<\infty$. Obviously $|f|_{\var}=0$ only when $f$ is constant, and every function of bounded variation is bounded and Borel measurable with at most countably many points of discontinuity. We let $\mathscr{BV}$ denote the set of all $f \colon [0,1] \to \mathbb{C}$ with bounded variation and equip it with the norm $\|f\|_{\mathscr{BV}}:=\|f\|_\infty + |f|_{\var}$ with respect to which it is a Banach space. We let $\mathscr{N} \subset \mathscr{BV}$ denote the subspace of functions which are nonzero at at most countably many points. We observe that $\mathscr{N}$ is a closed subspace: indeed, in the equivalent norm $\frac{1}{2}|\cdot|_{\var}$ the space $\mathscr{N}$ is isometrically isomorphic to the space of $L^1$ functions relative to counting measure on $[0,1]$. 

The following Liv\v{s}ic regularity theorem for the $\beta$-transformation removes an integrability condition occurring in the earlier works \cite{NiPe12,PoYu99}:
\begin{theorem}\label{th:app2}
Let $\beta>1$ and let $T_\beta \colon [0,1] \to [0,1]$ by given by $T_\beta x:=\beta x \mod 1$. If $f \in \mathscr{BV}$ satisfies $f = h \circ T - h$ Lebesgue-a.e. for some measurable function $h \colon [0,1] \to \mathbb{C}$, then there exists $g \in \mathscr{BV}$ such that $h=g$ Lebesgue a.e. 
\end{theorem}
Theorem \ref{th:app2} could of course be extended to broader classes of piecewise monotone maps by means of works such as \cite{HoKe82,Ry83}, and in higher dimensions such results as \cite{Li13,Sa00,Ts01} could also be applied. The literature on these broad classes of transformations however contains relatively few ``off-the-shelf'' results sufficient to give a clean and general proof of the hypothesis O\ref{it:eigen}, and we therefore confine our attention to the case of $\beta$-transformations on the basis of its simplicity and long interest in the literature.
\begin{proof}[Proof of Theorem \ref{th:app2}]
Let $X=[0,1]$ and let $\mu$ denote Lebesgue measure on that interval. By treating the real and imaginary parts of $f$ and $h$ separately we assume without loss of generality that $f$ and $h$ are real-valued. The non-singularity property O\ref{it:nonsing} is quite trivial, and O\ref{it:eigen} dates back to the work of Renyi \cite{Re57} where it is shown additionally that the density $\chi$ satisfies $\ess \inf \chi\geq 1-1/\beta>0$.  We define $\mathfrak{X}:=\mathscr{BV}/\mathscr{N}$. Since obviously every element of $\mathscr{N}$ has integral zero with respect to Lebesgue measure and therefore corresponds to the zero element of $L^1(\mu)$, the space $\mathscr{BV}/\mathscr{N}$ embeds continuously in $L^1(\mu)$. Since every characteristic function of an interval belongs to $\mathscr{BV}$ the range of this embedding is dense. If an element of $\mathscr{BV}$ is zero Lebesgue a.e. then in particular it is zero on a dense subset of $[0,1]$, and this is only possible if it belongs to $\mathscr{N}$; thus the embedding $\mathscr{BV}/\mathscr{N} \to L^1(\mu)$ is injective, and O\ref{it:embed} is satisfied. 
Define $\hat{\mathscr{L}} \colon \mathscr{BV} \to \mathscr{BV}$ by
\[(\hat{\mathscr{L}}\phi)(x):=\sum_{T_\beta y=x}\frac{\phi(y)}{\beta}.\]
It is easily demonstrated that $\hat{\mathscr{L}}$ is a bounded operator on $\mathscr{BV}$ and satisfies
\[\int(\psi \circ T)\cdot\phi \,d\mu = \int \psi \cdot (\hat{\mathscr{L}}\psi\circ T)d\mu\]
for every $\psi \in L^\infty(\mu)$ by an elementary change of variable. Clearly $\hat{\mathscr{L}}$ preserves $\mathscr{N}$ and therefore induces an operator $\mathscr{L}$ on $\mathfrak{X}$ which satisfies O\ref{it:transfer}. It follows from the results of \cite[\S3.2]{Ba00} that the essential spectral radius of $\mathscr{L}$ on $\mathfrak{X}$ is strictly less than $1$, which gives O\ref{it:fredholm}.  If  $f \in \mathscr{BV}$ satisfies $f = h \circ T - h$ Lebesgue-a.e. then its image in $\mathscr{BV}/\mathscr{N}$ obviously has the same property, which gives F\ref{it:cocycle}.  It is not difficult to see that $\mathscr{BV}$ is a Banach algebra with respect to pointwise multiplication from which F\ref{it:bdd} follows easily. Applying Theorem \ref{th:main} we find that the density $\chi$ coincides a.e. with an element of $\mathscr{BV}/\mathscr{N}$ and that there exists $g \in \mathscr{BV}/\mathscr{N}$ such that $g =h \chi$ a.e. Choose $\chi \in \mathscr{BV}$ and $g \in \mathscr{BV}$ to be representatives modulo $\mathscr{N}$ of the corresponding elements of $\mathscr{BV}/\mathscr{N}$ in  such a way that $\inf \chi >0$. Clearly $g=h \chi $ a.e. and $1/\chi \in \mathscr{BV}$, whence $h=g/\chi \in\mathscr{BV}$ a.e. as required.
\end{proof}

\subsection{Real-analytic expanding maps of the circle and bounded holomorphic cocycles}\label{ss:ex1}
Measurable Liv\v{s}ic regularity theorems are typically stated in the H\"older and $C^r$ classes on the basis that once continuous regularity has been obtained, higher regularity can be obtained subsequently by other methods. Theorem \ref{th:main} on the other hand straightforwardly allows one to proceed from measurable to analytic regularity in a single step, for example as follows. 

Let $S^1 \subset \mathbb{C}$ be the unit circle and for every pair of real numbers $r,R$ satisfying $0<r<1<R$ define
\[\mathcal{A}_{r,R}:=\{z \in \mathbb{C} \colon r<|z|<R\}.\]
Define also
\[\mathscr{A}:=\left\{\mathcal{A}_{r,R} \colon 0<r<1<R\right\}.\]
For every $\mathcal{A}\in\mathscr{A}$ let $H^\infty(\mathcal{A})$ denote the complex Banach space of bounded holomorphic functions $\mathcal{A} \to \mathbb{C}$ equipped with the supremum norm. Let $\mu$ denote normalised Lebesgue measure on $S^1$. The following proposition summarises the results of \cite[\S2]{SlBaJu13}:
\begin{proposition}\label{pr:slartibardfast}
Suppose that $T \colon S^1 \to S^1$ is a real analytic map which satisfies $\inf_{z \in S^1}|T'(z)|>1$ and let $U \subseteq \mathbb{C}$ be an open set such that $S^1 \subset U$. Then there exist $\mathcal{A} \in \mathscr{A}$ satisfying $\mathcal{A} \subseteq U$ and a compact linear operator $\mathscr{L} \colon H^\infty(\mathcal{A}) \to H^\infty(\mathcal{A})$ such that
\[\int_{S^1} \psi \cdot (\mathscr{L}\phi)\,d\mu = \int_{S^1} (\psi \circ T)\cdot\phi\,d\mu\]
for all $\phi \in H^\infty(\mathcal{A})$ and $\psi \in L^1(\mu)$.
\end{proposition}
We easily deduce:
\begin{theorem}\label{th:app1}
Let $T \colon S^1 \to S^1$ be a real-analytic map satisfying $\inf_{z \in S^1}|T'(z)|>1$, $U \subset \mathbb{C}$ an open neighbourhood of $S^1$, and $f \colon U \to \mathbb{C}$ a holomorphic function such that $f = h \circ T -h$ $\mu$-a.e. for some measurable function $h \colon S^1 \to \mathbb{C}$. Then there exist an open neighbourhood $V$ of $S^1$ and a holomorphic function $g \colon V \to \mathbb{C}$ such that $g=h$ $\mu$-a.e.
\end{theorem}
\begin{proof}
By shrinking the neighbourhood $U$ if necessary we may assume that it is invariant under the transformation $z \mapsto 1/z^*$. Define a holomorphic function $\hat{f} \colon U \to \mathbb{C}$ by $\hat{f}(z):=f(1/z^*)^*$ for all $z \in U$, then we may write $f = f_1+if_2$ where $f_1:=\frac{1}{2}(f+\hat{f})$ and $f_2:=-\frac{i}{2}(f-\hat{f})$. We observe that both $f_1$ and $f_2$ are holomorphic and take only real values on $S^1$. By applying Theorem \ref{th:app1} to each of $f_1$ and $f_2$ separately we may therefore solve the cohomological equation for $f$ under the additional assumption that $f$ takes only real values on $S^1$. We make this assumption for the remainder of the proof without loss of generality.

Let $\mu$ denote normalised Lebesgue measure on $S^1$. Since $T$ is a local diffeomorphism of $S^1$ it is nonsingular with respect to  $\mu$, which gives O\ref{it:nonsing}. Let $f \colon U \to \mathbb{C}$ be a holomorphic function which takes only real values on $S^1$, and which satisfies $f = h \circ T -h$ $\mu$-a.e, where we assume without loss of generality that $h$ is real-valued. By shrinking $U$ if necessary we may assume without loss of generality that $f$ is bounded on $U$. 

The existence of a unique absolutely continuous $T$-invariant probability measure is classical (see for example \cite[Corollary 5.1.25]{KaHa95}) which yields O\ref{it:eigen}, and by \cite[Theorem 5.1.16]{KaHa95} the density $\chi\colon S^1 \to \mathbb{R}$ is continuous and positive. Let $\mathcal{A} \subset U$ and $\mathscr{L}$ be as given by Proposition \ref{pr:slartibardfast}. It is obvious that $H^\infty(\mathcal{A})$ embeds continuously in $L^1(\mu)$, and since it contains the restriction to $\mathcal{A}$ of every complex polynomial it embeds densely in $L^1(\mu)$. If $\phi \in H^\infty(\mathcal{A})$ is zero $\mu$-a.e. then it is zero on $S^1$ by continuity and hence is the zero element of $H^\infty(\mathcal{A})$ by analyticity and connectedness, so the embedding $H^\infty(\mathcal{A}) \to L^1(\mu)$ is injective and we have demonstrated O\ref{it:embed}. Hypothesis O\ref{it:transfer} is given in Proposition \ref{pr:slartibardfast} and hypothesis O\ref{it:fredholm} is obvious by the compactness of $\mathscr{L}$. Since $f$ is holomorphic and bounded on $U$ it belongs to the Banach algebra $ H^\infty(\mathcal{A})$, from which F\ref{it:bdd} follows. Clearly F\ref{it:cocycle} is also satisfied. Applying Theorem \ref{th:main} it follows that $\chi$ extends to an element of $H^\infty(\mathcal{A})$ and that there exists $\hat{g} \in H^\infty(\mathcal{A})$ such that $\hat{g}=h\chi$ a.e. 
Taking $V\subset \mathcal{A}$ to be an open neighbourhood of $S^1$ which contains no zeros of $\chi$, we deduce that $h$ is a.e. equal to the function $\hat{g}/\chi$ which is holomorphic on $V$.\end{proof}

\subsection{Virtually expanding maps and H\"older cocycles}
Let $T \colon M \to M$ be a smooth covering map of a connected closed smooth Riemannian manifold $M$ with nowhere vanishing Jacobian determinant. Following M. Tsujii in \cite{Ts23} we say that $T$ is \emph{$\tau$-virtually expanding} if for some (then for all sufficiently large) $n \geq 1$,
\[\sup_{\substack{(x,v)\in \mathcal{T}^*M \\ v\neq0}} \sum_{T^ny=x} |JT^n(y)|^{-1} \left(\frac{\|((D_xT^n)^*)^{-1}v\| }{\|v\|}\right)^s<1, \]
where $\mathcal{T}^*M$ denotes the cotangent bundle of $M$ and $JT^n$ the Jacobian determinant of $T^n$. If $\dim M>1$ then for every $s>0$ there exist $s$-virtually expanding maps which are not expanding. As an example, Tsujii shows that for any $s>0$, if $m$ is large enough the map $T \colon \mathbb{R}^2/\mathbb{Z}^2 \to \mathbb{R}^2/\mathbb{Z}^2$ defined by
\[T(x,y):= (mx, y+m\cos 2\pi x)\]
is $s$-virtually expanding. We observe that for this particular map Lebesgue measure is invariant and ergodic by classical results on the ergodicity of skew-products, making it an instance of the following theorem:
\begin{theorem}\label{th:tsujii}
Let $T \colon M \to M$ be $s$-virtually expanding, where $2s>d:=\dim M$. Suppose that $T$ has a unique absolutely continuous invariant measure and that the density of this measure is bounded away from $0$. If $f \in C^s(M)$ satisfies $f=h\circ T - h$ Lebesgue a.e, then there exists $g \in  C^{s-d/2}(M)$ such that $f=g\circ T-g$ Lebesgue a.e.
\end{theorem}
The drop in regularity from $C^s(M)$ to $C^{s-d/2}(M)$ arises via the Sobolev embedding theorem. We expect that this loss of regularity could be reduced or removed by an extension of Tsujii's results from the context of the Sobolev Hilbert space $H^\tau(M)$ to Sobolev spaces $W^{\tau,p}(M)$ with $p> 2$, but we do not attempt such an extension here.
\begin{proof}
Let $\mu$ denote normalised Lebesgue measure on $M$. Since $T$ is a local diffeomorphism it is non-singular with respect to $\mu$, which yields O\ref{it:nonsing}, and we have presumed O\ref{it:eigen} in the hypotheses of the theorem.

We review some facts concerning Sobolev spaces of fractional index (see for example \cite{Ta81}). For each real $\tau\geq 0$ let $H^{\tau}(M)$ denote the Sobolev space of order $\tau$ on the manifold $M$. Here the Sobolev space $H^\tau(\mathbb{R}^d)$ may be understood as the completion of the vector space of all smooth compactly supported functions $\mathbb{R}^d \to \mathbb{C}$  with respect to the norm
\[\|\phi\|_{H^\tau(\mathbb{R}^d)}:=\left(\int_{\mathbb{R}^d} \left(1+\|\xi\|^2\right)^{\tau} \left\|\hat{\phi}(\xi)\right\|^2 d\xi\right)^{\frac{1}{2}},\]
where $\hat{\phi}$ denotes the Fourier transform of $\phi$; the Sobolev space $H^\tau(M)$ is then constructed via a representation of $M$ using a fixed smooth partition of unity in a natural fashion. Up to equivalence of norms the space thus defined is independent of the choice of partition of unity. Both $H^\tau(\mathbb{R}^d)$ and $H^\tau(M)$ are Hilbert spaces, and when $\tau>d/2$ they are Banach algebras; the latter contains $C^\tau(M)$ as a linear subspace for all $\tau \geq 0$. As a consequence of the Sobolev embedding theorem, $H^\tau(M)$ embeds continuously in $C^{\tau-d/2}(M)$ when $\tau> d/2$. 

Define $\mathfrak{X}:=H^s(M)$, which as just noted embeds continuously in $C^{s-d/2}(M)$ and hence in particular embeds continuously in $L^1(\mu)$; since $H^s(M)$ contains $C^s(M)$ this embedding is also dense,  so O\ref{it:embed} is satisfied. The operator $\mathscr{L} \colon H^s(M) \to H^s(M)$ defined by
\[(\mathscr{L}\phi)(x)=\sum_{Ty=x} \frac{\phi(y)}{|JT(y)|}\]
satisfies O\ref{it:transfer} by standard calculations and is shown in \cite{Ts23} to have essential spectral radius strictly less than $1$, yielding O\ref{it:fredholm}.  By hypothesis $f \in C^s(M)\subset \mathfrak{X}$ satisfies F\ref{it:cocycle}, and  the multiplication axiom F\ref{it:bdd} follows since $H^s(M)$ is a Banach algebra. It follows from Theorem \ref{th:main} that $\chi \in H^s(M)$ and there exists $g \in H^s(M)$ such that $h\chi = g$ Lebesgue a.e. In particular $g, \chi \in C^{s-d/2}(M)$ and by hypothesis $\chi$ can have no zeros on $M$, so we have $h=g/\chi$ a.e. with $g/\chi \in C^{s-d/2}(M)$ as claimed.
\end{proof}

\section{Acknowledgments}
]

The author thanks Oliver Butterley, Peyman Eslami, Carlangelo Liverani and Arick Shao for helpful conversations relating to this work.

\bibliographystyle{acm}
\bibliography{ML}

\end{document}